\documentclass[draft]{article}
\usepackage{amsmath}
\usepackage{amsfonts}
\usepackage[margin=1in]{geometry}

\newtheorem{theorem}{Theorem}

\newtheorem{lemma}[theorem]{Lemma}
\newtheorem{proposition}[theorem]{Proposition}
\newtheorem{conjecture}[theorem]{Conjecture}
\title{The image of multilinear polynomials evaluated on $3\times 3$ upper triangular matrices\thanks{Supported by grant \#2018/15627-2, S\~ao Paulo Research Foundation (FAPESP).}}
\author{Thiago Castilho de Mello\thanks{Universidade Federal de S\~ao Paulo -- Instituto de Ci\^encia e Tecnologia, tcmello@unifesp.br}
}

\date{}
\begin{document}
\maketitle

%\noindent\textbf{Keywords:} generic algebras; basis of polynomial identities; PI equivalence; matrices over Grassmann algebras.

%
%\noindent\textbf{2010 AMS MSC Classification:} 16R10, 16R20, 16R99

\begin{abstract}
	We describe the images of multilinear polynomials of arbitrary degree evaluated on the $3\times 3$ upper triangular matrix algebra over an infinite field.
	
	
\end{abstract}

\section{Introduction}

A famous open problem known as Lvov-Kaplansky's conjecture asserts: the image of a multilinear polynomial in noncommutative variables on the matrix algebra $M_{n}(K)$ over a field $K$ is a vector space \cite{Dniester}.

A major breakthrough in this direction was made by Kanel-Belov, Malev and Rowen \cite{Kanel2, Malev}, with the solution of the conjecture for $n=2$. Also for $3\times 3$ matrices the authors obtained significant results \cite{Kanel3}, but the complete problem for matrices of order $\geq 3$ is still open.
For the state-of-art of the Lvov-Kaplansky conjecture see \cite{Survey}.

This conjecture motivated many different studies related to images of polynomials. For instance, papers on images of Lie, and Jordan polynomials, and also for other algebras have been published since then. The particular case of $n\times n$ upper triangular matrices over a field $K$, $UT_n(K)$ (or simply $UT_n$), was studied by the author and Fagundes in \cite{FagundesdeMello}, where the Lvov-Kaplansky's conjecture for $UT_n(K)$ was proved to be equivalent to the following

\begin{conjecture}
	The image of a multilinear polynomial on $UT_n(K)$ is either $\{0\}$, $UT_n$ or $J^{k}$ for some integer $k\geq 1$, where $J=J(UT_n)$ is the Jacobson radical of $UT_n$, i.e., the set of strictly upper triangular matrices.
\end{conjecture}

We remark that in the papers $\cite{Fagundes, FagundesdeMello}$ $J^k$ was denoted by $UT_n^{(k-1)}$. This is the set of matrices whose entries $(i,j)$ are zero if $i\geq j+k$, for $k\geq 1$.

A general result in the direction was given in $\cite{Fagundes}$, where the author proves the conjecture for polynomials of arbitrary degree evaluated on strictly upper triangular matrices.

In \cite{FagundesdeMello}, the authors prove the conjecture for polynomials of degree up to 4 over fields of zero characteristic. 

For polynomials of arbitrary degree, the problem was solved in the case $n=2$ by Fagundes in his master's degree dissertation \cite{FagundesDissertation}, but the text was written in Portuguese and the result was not published elsewhere. Some time latter, in \cite{Wang} (see also \cite{WangCorr} for a correction of the paper), Wang gives another proof of this result. 

In this short note, we present Fagundes' proof of the above result, since it is extremely simpler then Wang's proof, and we prove the above conjecture for $n=3$.

The techniques here used are based on the paper \cite{FagundesdeMello} and on the theory of algebras with polynomial identities (PI-algebras). We suppose the reader has the basic knowledge of this subject. For more information and notation, see the book \cite{drenskybook}. We believe the technique used in this paper may be adapted to solve the general problem for $UT_n$. This will be considered in a future project.

\section{Preliminary Results}

In this section we recall some results on the theory of PI-algebras and results on images of multilinear polynomials given in \cite{FagundesdeMello}. We denote by $K\langle X \rangle$ the free associative algebra freely generated by the set $X$, i.e., the algebra of noncommutative polynomials in the variables of $X$, and by $P_n$ the set of multilinear polynomials of degree $n$ in $K\langle X \rangle$. If $S\subseteq K\langle X \rangle$, we denote by $\langle S \rangle ^T$, the $T$-ideal generated by $S$.

\begin{theorem}[Theorem 5.2.1 of \cite{drenskybook}]\label{Dr}
	Let $K$ be an infinite field and let $UT_n(K)$ be the algebra of $n\times n$ upper triangular matrices over $K$. 
	
	\begin{enumerate}
		\item The polynomial identity
		\[ [x_1, x_2]\cdots [x_{2n-1}, x_{2n}] = 0\]
		forms a basis of the polynomial identities of $UT_n(K)$.		
		\item The relatively free algebra $F(UT_n(K))$ has a basis consisting of all products
		$$
		x_1^{a_1}\ldots x_m^{a_m}[x_{i_{11}},x_{i_{21}},\ldots,x_{i_{p_1}1}]\ldots[x_{i_{1r}},x_{i_{2r}},\ldots,x_{i_{p_rr}}],
		$$
		where the number $r$ of participating commutators is $\le n-1$ and the indexes in each commutator $[x_{i_{1s}},x_{i_{2s}},\ldots,x_{i_{p_ss}}]$ satisfy $i_{1s}>i_{2s}\le\cdots\le i_{p_ss}$.
		
	\end{enumerate}
\end{theorem}

%\begin{remark}
%
%	If $2k\leq n$, We have the following chain of subspaces of $Pn$:
%	
%	\[P_n\supseteq P_n\cap \langle [x_1,x_2]\rangle^T %\supseteq P_n\cap \langle [x_1,x_2][x_3,x_4]\rangle^T
%	 \supseteq \cdots \supseteq P_n\cap \langle [x_1,x_2][x_3,x_4]\cdots [x_{2k-1},x_{2k}]\rangle^T \]	
%
%In particular, if $f\in P_n$, there exists an index $s$ such that $f\in P_n\cap \langle [x_1,x_2][x_3,x_4]\cdots [x_{2s-1},x_{2s}]\rangle^T$ but $f\not \in P_n\cap \langle [x_1,x_2][x_3,x_4]\cdots [x_{2s-3},x_{2s-2}]\rangle^T$.
%
%\end{remark}

The following is an improvement of \cite[Proposition 8]{FagundesdeMello} and is easy to prove.

\begin{lemma}\label{8}
	Let $f(x_1,\dots,x_n)=\sum_{\sigma\in S_n} \alpha_{\sigma} x_{\sigma(1)} \cdots x_{\sigma(n)}$. Then 
	
	\begin{enumerate}
		\item If $\sum_{\sigma \in S_n}a_{\sigma}\neq 0$, then the image of $f$ on a unitary algebra $A$ is $A$.
		
		\item $f\in \langle [x_1,x_2]\rangle ^T$ if and only if $\sum_{\sigma \in S_n} a_{\sigma}=0$.
	\end{enumerate}
\end{lemma}

\begin{lemma}[Lemma 10 of \cite{FagundesdeMello}]\label{10}
	Let $K$ be a field with at least $n$ elements and let $d_{1},\dots,d_{n}\in K$ be distinct elements. Then for $D=diag(d_{1},\dots,d_{n})$ and $k\geq 0$, we have 
	\begin{eqnarray}\nonumber
	[J ^k,D]=J^k  \ \mbox{and}\  [UT_{n},D]=J.
	\end{eqnarray}
\end{lemma}

\section{The case n=2}

In this section we give Fagundes' proof \cite[Proposition  4.7]{FagundesDissertation} of the case $n=2$ of the above conjecture. This is a simple application of Lemma \ref{8}.

\begin{proposition}
	Let $K$ be a field. The image of a multilinear polynomial on $UT_2$ is $UT_2$, $J$, or $\{0\}$.
\end{proposition}

\begin{proof}
Let $f(x_1,\dots,x_m)=\sum_{\sigma \in S_n} \alpha_{\sigma} x_{\sigma(1)} \cdots x_{\sigma(n)}$ be a multilinear polynomial. If $\sum_{\alpha_{\sigma}}\neq 0$, then the image of $f$ on $UT_2$ is  $UT_2$.
Suppose $\sum_{\alpha_{\sigma}}=0$. Then $Im(f)\subseteq J$. Now if $f$ is a polynomial identity, the image of $f$ is $\{0\}$. Otherwise, since $J$ is a one-dimensional subspace of $UT_2$, and the image of $f$ is closed under scalar multiplication, the image of $f$ is $J$.% The above proves the following
\end{proof}

\section{The case n=3}

\begin{theorem}
	Let $K$ be an infinite field and $f\in P_n$. The image of $f$ in $UT_3$ is either $UT_3$, $J$, $J^2$ or $\{0\}$.
\end{theorem}

\begin{proof}
	Write $f(x_1,\dots,x_n)=\sum_{\sigma\in S_n} \alpha_{\sigma} x_{\sigma(1)} \cdots x_{\sigma(n)}$.
	If $f$ is a polynomial identity for $UT_3$, then the image of $f$ is $\{0\}$. So from now on we suppose $f$ is not an identity for $T_3$.

%	Let us suppose $f$ is not a polynomial identity for $UT_3$.
	
	If $\sum_{\sigma \in S_n}a_{\sigma}\neq 0$, then item 1 of the Lemma \ref{8} asserts that the image of $f$ is $UT_3$.
	
	So let us suppose that $\sum_{\sigma \in S_n}a_{\sigma}=0$. Item 2 of Lemma \ref{8} yields  $f\in \langle [x_1,x_2]\rangle ^{T}$. As a consequence, the image of $f$ is a subset of $J$.
	
	By Theorem \ref{Dr}, up to an identity of $UT_3$, $f$ can be written as a linear combination of the following two families of polynomials:
	
	\begin{enumerate}
		\item $x_{i_1}\cdots x_{i_k}[x_{i_{k+1}},\dots,x_{i_{k+t}}]$, with $k+t=n$,  $i_1<i_2<\dots <i_k$ and $i_{k+1}>i_{k+2}<i_{k+3}< \cdots < i_{k+t}$
		
		\item $x_{i_1}\cdots x_{i_k}[x_{i_{k+1}},\dots,x_{i_{k+t}}][x_{i_{k+t+1}},\dots,x_{i_{k+t+s}}]$, with $k+t+s=n$,\\ $i_1<i_2<\dots <i_k$ , 
		$i_{{k+1}}>i_{{k+2}}<i_{{k+3}}< \cdots < i_{{k+t}}$ and
		$i_{k+t+1}>i_{k+t+2}<i_{k+t+3}< \cdots < i_{k+t+s}$. 
	\end{enumerate}

	Let us first assume that $f\not \in \langle [x_1,x_2][x_3,x_4]\rangle ^T$. So there exists at least one polynomial of the first family of polynomials above with nonzero coefficient in the decomposition of $f$ as a linear combination of such polynomials.
	
	Let us consider one such polynomial with $t$ minimal. Now perform the substitution $x_{i_j}\mapsto I$, for $j=1,\dots,k$. Here $I$ stands for the $3\times 3$ identity matrix.
	Then, after this substitution, the polynomial $f$ is a nonzero linear combination of polynomials $[x_{j_1},\dots,x_{j_t}]$ with $j_1>j_2<j_3<\cdots<j_t$ and some polynomials in $\langle [x_1,x_2][x_3,x_4]\rangle ^{T}$.
	
	Now let $d_1,d_2,d_3$ be distinct elements in $K$ and $D=diag(d_1,d_2,d_3)$. By Lemma \ref{10} $[UT_n,D]=J$ and $[J,D]=J$. As a consequence,  $[UT_n,D,\dots,D]=J$. After substituting all variables but one by $D$, the polynomials in $\langle [x_1,x_2][x_3,x_4]\rangle ^{T}$ vanish. All the other polynomials but one also vanish, and we obtain that the image of $f$ is $J$. 
	
	To finish the proof, we assume $f\in \langle[x_1,x_2][x_3,x_4] \rangle^T$. Of course, the image of $f$ is a subset of $J^2$, which is 1-dimensional. Since $f$ is not a polynomial identity for $UT_3$, its image is $J^2$ and the theorem is proved.
	
\end{proof}

%\section{Funding}
%The author was supported by grant \#2018/15627-2, S\~ao Paulo Research Foundation (FAPESP).

\end{document}